\documentclass[12pt]{article}
\usepackage{amssymb,amsmath,amsfonts}
\usepackage{amsthm}
\usepackage{multirow}
\usepackage{epsfig}
\usepackage{subfigure}
\usepackage{color}

\usepackage{mathabx}

\numberwithin{equation}{section}
\theoremstyle{plain}
\newtheorem{theorem}{Theorem}[section]
\newtheorem{corollary}[theorem]{Corollary}

\theoremstyle{definition}
\newtheorem{definition}{Definition}[section]
\newtheorem{example}{Example}[section]
\theoremstyle{remark}
\newtheorem{remark}{\rm\bf Remark}[section]

\allowdisplaybreaks

\begin{document}

\title{Power Lindley distribution and software metrics}
\author{Mohammed Khalleefah, Sofiya Ostrovska and Mehmet Turan}
\date{}
\maketitle

\begin{center}
{\it Atilim University, Department of Mathematics, Incek  06836, Ankara, Turkey}\\
{\it e-mail: mas\_libya86@yahoo.com, sofia.ostrovska@atilim.edu.tr, mehmet.turan@atilim.edu.tr}\\
{\it Tel: +90 312 586 8211,  Fax: +90 312 586 8091}
\end{center}

\begin{abstract}
The Lindley distribution and its numerous generalizations are widely used in statistical and engineering practice. Recently, a power transformation of Lindley distribution, called the power Lindley distribution, has been introduced by M. E. Ghitany et al., who initiated the investigation of its properties and possible applications. In this article, as a continuation of the preceding research, new results on the power Lindley distribution are presented. The focus of this work is on the moment-(in)determinacy of the distribution for various values of the parameters. Afterwards, certain applications are provided to describe data sets of software metrics.

\end{abstract}



\section{Introduction }

Nowadays, new families of probability distributions are being proposed by a large number of authors with the aim to provide appropriate tools to study the tendencies in the behavior of 
data sets emerging in the financial mathematics, medical research, computer science, engineering, and other disciplines. See, for example, \cite{newweibull,ghitanypower,koutras}. Using a variety of criteria and approaches, researchers are seeking distributions to best match experimental data.

The Lindley distribution was introduced in 1958 by D. V. Lindley \cite{lindley}. 
Yet, it continues to draw attention from mathematics and its applications, giving rise to new extensions and modifications. See, for example, \cite{ wind, extended, ghitanylindley, ghitanyzero, ghitanytwo}.
The \textit{Lindley distribution} with parameter $\beta >0$ is defined by the probability density function (PDF) of the form:
\begin{equation}\label{lind}
f(x)=\frac{\beta^2}{\beta+1}(1+x)e^{-\beta x}, \quad x>0.
\end{equation}
Formula \eqref{lind} shows that the Lindley distribution is a two-component mixture of the exponential and two-stage Erlang distributions with the mixing proportion $p=\beta/(\beta+1).$ The distributions of this form come out in reliability theory, for example, in the study of imperfect fault coverage with the probability $p$ of the replacement failure. A comprehensive study of the Lindley distribution and its applications in the framework of reliability theory is performed in \cite{ghitanylindley}. It can be observed that the Lindley distribution as well as the gamma distribution belong to the family of  Kummer distributions. The latter was first introduced in 1993 by Armero and Bayarri for conducting a statistical analysis of $M/M/\infty$ systems. See \cite{armero93, armero}. The study of the Kummer distribution was followed up in \cite{ng} by K. W. Ng and S. Kotz, who obtained new results on the subject and expanded the assortment of the Kummer-type distributions.
The current paper deals with the properties and applications of the power Lindley distribution, which represents the class of $p$-Kummer distributions introduced in \cite{kuwait}. The power Lindley distribution was put forward in 2013 by Ghitany et al.
as follows.

\begin{definition}
 $\cite{ghitanypower}$ The \textit{power Lindley} distribution with parameters $\alpha, \beta >0$ is defined by its PDF function:
 \begin{equation}\label{powerlind} f(x)= \frac{\alpha \beta^2}{\beta +1}\left(1+x^\alpha\right)x^{\alpha -1}e^{-\beta x^\alpha},\quad x>0.\end{equation}
\end{definition}

We write $X\sim PL(\alpha, \beta)$ to indicate that a random variable $X$ possesses a power Lindley distribution with parameters $\alpha$ and $\beta$. Evidently, when $\alpha =1,$ one recovers a Lindley distribution with PDF \eqref{lind}. Observe that $X$ has a Lindley distribution with parameter $\beta$ if and only if $X^{1/\alpha}\sim PL(\alpha, \beta).$ That is, the power Lindley distribution occurs naturally as a power transformation of a random variable following Lindley distribution. Along with that,  power Lindley distribution can also be viewed as a particular case of the $p$-Kummer distribution, whose PDF is given by:
\begin{equation*} 
f_p(x)=\frac{x^{a/p-1}\left(1+x^{1/p}\right)^{-c}\exp\left(-b x^{1/p}\right)}{p\Gamma (a) U(a-c +1,b)},\;a, b, p>0, \;c \in \mathbb{R}, x>0.
\end{equation*}
See \cite[Definition 2]{kuwait}. Here, $\Gamma$ is Euler's gamma-function and $U$ is Kummer's function of the second kind. For their definitions and properties, one may refer to \cite[formulae  6.1.1 and 13.1.3]{abramow}.
Obviously, $X\sim PL(\alpha,\beta)$ if and only if it has $p$-Kummer distribution with $p= 1/\alpha$ and the parameters 
$a=1,$ $b=\beta$ and $c=-1.$

This paper aims to pursue the study of the power Lindley distribution initiated in \cite{ghitanypower}. Specifically, the moment-(in)determinacy for different values of parameters will be determined. It has to be noticed that the moment-(in)determinacy of a probablity distribution is an important factor not only in probability theory, but also in applied areas, see \cite{aerosol, penson, stoyanovfinance}. Moreover, the increasing role of heavy-tailed distributions in financial, engineering and computer science research (\cite{ferreira, mile, stoyanovfinance}) puts additional weight on this subject.  In this connection,  exemplary Stieltjes classes for power Lindley distributions will be provided in the event of the moment-indeterminacy. Finally, some applications will be given to the data sets of software metrics.

\section{Main results}

It is known (\cite[p, 497]{ghitanylindley}) that the characteristic function of the Lindley distribution is expressed by:
$$\phi(t) = \frac{\beta^2(\beta+1-it)}{(\beta+1)(\beta-it)^2}$$ and hence it is analytic for $t\in (-\beta,\beta), $ implying that
the Lindley distribution is moment-determinate.
The situation with the power Lindley distribution is less straightforward, since, for $\alpha <1,$ the characteristic function of $PL(\alpha, \beta)$ distribution is not analytic at 0. Theorem \ref{th2} presents a necessary and sufficient  condition for the moment-(in)determinacy of the power Lindley distribution.

To begin with, some analytical properties of the characteristic functions of the power Lindley distribution are stated in the next claim.

\begin{theorem}\label{th1}
The characteristic function of a power Lindley distribution is entire of order $\alpha/(\alpha-1)$ when $\alpha>1$, analytic on interval $(-\beta, \beta)$ when $\alpha=1$, and is not analytic at 0 otherwise.
\end{theorem}

\begin{proof}
The conditions for the analyticity of the characteristic function can be expressed in terms of the tail function, which for the power Lindley distribution coincides with its survival function $S(x)$. According to \cite[formula (3)]{ghitanypower}:
$$S(x)=\left(1+\frac{\beta}{\beta+1}x^\alpha\right)e^{-\beta x^\alpha}, x>0.$$ By \cite[formula (2.2.3)]{linnik}, the characteristic function of the distribution is analytic on $(-R,R)$ if and only if its tail function satisfies
\begin{equation}\label{ent}S(x)=O\left(e^{-rx}\right), \quad x\rightarrow\infty \;\;\mathrm{for\;\;each}\;\;r<R.\end{equation}
Clearly, for $\alpha >1,$ condition \eqref{ent} holds for all $R>0,$ whence in this case the characteristic function is entire, while  for $\alpha=1,$ estimate \eqref{ent} is true only when $r<\beta.$ As for $\alpha <1,$ condition \eqref{ent} is violated whatever $R$ is and, therefore, the characteristic function is not analytic at 0. In the case of the entire characteristic function, its  order $\rho$ and type $\sigma$ can be calculated by Theorem 2.4.4 of 
\cite{linnik}, yielding $\rho = \alpha/(\alpha-1),$ $\sigma=\frac{\alpha -1}{\alpha}\left(\alpha\beta\right)^{-1/(\alpha-1)},$ respectively.
\end{proof}

\begin{corollary} The outcomes of Theorem \ref{th1} can be restated in the following way. The moment generating function of the power Lindley distribution with parameters $\alpha$ and $\beta$:

\begin{itemize} \item exists for all real numbers if $\alpha >1;$

\item exists on interval $(-\beta,\beta)$ if $\alpha=1;$

\item does not exist if $\alpha <1.$ \end{itemize}

\end{corollary}

\begin{corollary} If $\alpha\geq 1,$ then $PL(\alpha,\beta)$ distribution is moment-determinate.\end{corollary}

This follows immediately from Cram\'er's condition for the moment-determinacy \cite[Theorem 1]{recent}.
The case $\alpha <1$ needs an additional investigation. Notice that in this case the distribution $PL(\alpha,\beta)$ becomes heavy-tailed. While each light-tailed distribution is uniquely determined by its moments,  for heavy-tailed distributions the uniqueness may not hold.  Heavy-tailed distributions, many of which are not unique with respect to the moments, are instrumental in stock market modeling and engineering \cite{stoyanovfinance}. 
For this reason, non-uniqueness of the distributions with respect to moments needs deep investigation.
The respective outcomes on the moment-(in)determinacy of the power Lindley distribution are summarized in the next assertion.

\begin{theorem}\label{th2} 
The power Lindley distribution is moment-indeterminate if and only if $\alpha <1/2.$
\end{theorem}

\begin{proof} In essence, the proof is based on the estimates for the rate of growth of moments. The needed facts are presented in the review \cite{recent}. In the context of this proof, letter $C$ - with or without subscripts - is used to denote positive constant whose value does not need being evaluated.

If $X\sim PL(\alpha,\beta),$ then the moments of $X$ have been calculated in \cite{ghitanypower} as follows:
\begin{equation}\label{moments} m_k=\mathbf{E}\left[X^k\right]=\frac{k\Gamma(k/\alpha)[\alpha(\beta+1)+k]}{\alpha^2\beta^{k/\alpha}(\beta+1)}, \quad k\in \mathbb{N}.\end{equation}
Hence, if $\alpha\geqslant 1/2,$ then
$$\frac{m_{k+1}}{m_k}\leqslant C\frac{\Gamma(k/\alpha+2)}{\Gamma(k/\alpha)}=C(k/\alpha+2)(k/\alpha+1)=O\left(k^2\right),\quad k\rightarrow \infty,$$ and by the condition (s1) \cite[Theorem 2]{recent}, the distribution is moment-determinate.

To examine the case $\alpha <1/2,$ we write using \eqref{moments}:
$$m_k=C\frac{\Gamma (k/\alpha +1) (k/\alpha +\beta+1)}{\beta^{k/\alpha}}.$$ 
Applying Stirling's formula, one has:
\begin{equation}\label{mgrowth}
m_k=C\frac{(k/\alpha)(k/\alpha+\beta+1)\Gamma(k/\alpha)}{\beta^{k/\alpha}}\geqslant Ck^{3/2}\exp\left\{\frac{k}{\alpha}\ln k-C_1k\right\}\:\mathrm{for\;\;some}\;\;C,C_1>0.
\end{equation} 
Since $\alpha < 1/2,$ writing $1/\alpha=2+2\varepsilon,$ one obtains:
$$m_k\geq C_2k^{3/2}\exp\{(2+\varepsilon)k\ln k\}\geq C_2k^{(2+\varepsilon)k} \quad\mathrm{for\;\;all}\;\;k\in\mathbb{N}.$$
To show that the distribution is moment-indeterminate, the estimate \eqref{mgrowth} has to be supplemented by checking whether the density \eqref{powerlind} satisfies Lin's condition, that is, to show that Lin's function $L_f(x):=-\frac{xf^\prime(x)}{f(x)}$ is monotone increasing for $x>x_0$ and that $\displaystyle \lim_{x\rightarrow\infty}L_f(x)=+\infty.$ Plain calculations yield:
$$L_f(x)=-x\ln C-\alpha x^\alpha /(1+x^\alpha)-(\alpha -1)+\beta x^{\alpha +1}\sim \beta x^{\alpha +1}
\rightarrow+ \infty\quad \mathrm{as}\;\;x\rightarrow\infty.$$ 
In addition,
$$L_f^\prime(x)=\beta(\alpha +1)x^\alpha[1+o(1)]\quad\mathrm{as}\quad x\rightarrow \infty,$$ implying that $L_f^\prime(x)>0$ for $x$ large enough. Thus, by  \cite[Theorem 7]{recent} when $\alpha<1/2,$ distribution $PL(\alpha,\beta)$ is moment-indeterminate. The proof is complete.

\end{proof}

\begin{remark} Alternatively, the moment-(in)determinacy of a power Lindley distribution can be derived from \cite[Theorem 7]{kuwait}, where a more complicated approach was used.
\end{remark}

When a probability distribution is moment-indeterminate, the problem arises to expose different distributions with the same moments of all orders.
In this paper, this will be done by presenting Stieltjes classes for the density \eqref{powerlind}, which are infinite families of PDFs having the same moments of all orders. Although the Stieltjes classes per se can be traced to the works of P. L. Chebyshev, T. Stieltjes, and C. Heyde \cite{stieltjes, jap},  the name itself is quite recent.
To pay tribute to the contribution of Stieltjes to the moment problem, J. Stoyanov \cite{jap}  in 2004 suggested  the name `Stieltjes classes', thus triggering their systematic study, which is still in progress. See, for example \cite{recent, alea, pakes, penson} and references therein.

For the convenience of readers, we supply the necessary definitions below.

\begin{definition}
Let $f(x)$ be a PDF of a random variable $X$ with finite moments of all orders, and let $h(x)$ be an integrable function on $(-\infty,\infty)$ such that $\sup\limits_{x\in\mathbb{R}}|h(x)|=1.$ If, for all $k\in\mathbb{N}_{0},$
$$\int_{\mathbb{R}}x^{k}h(x)f(x)dx=0,$$
then $h(x)$ is called a {\it perturbation} function of the density $f(x).$
\end{definition}

\begin{definition}
Let $f(x)$ be a PDF and $h(x)$ be a perturbation function of $f(x).$ The set
$$S=S(f,h):= \{f_{\epsilon}(x):f_{\epsilon}(x)=f(x)[1+\epsilon h(x)],\ x\in\mathbb{R},\ \epsilon\in[-1,1]\}$$
is said to be a {\it Stieltjes class} for $f(x)$ based on $h(x).$
\end{definition}
Obviously, $S$ is an infinite family of densities all having the same sequence of moments as $f(x).$ Observe that, for a density function $f(x),$ there are different Stieltjes classes based on various perturbation functions $h(x).$
The next statement provides exemplary perturbation functions for \eqref{powerlind}.

\begin{theorem} The following functions are perturbations for PDF \eqref{powerlind} in the case $\alpha <1/2$:
\begin{enumerate} \item[{\rm(i)}] 
$\displaystyle H_1(x)=M_1\frac{x^{1-\alpha}}{1+x^\alpha}\exp(-\beta x^\alpha)\sin\left[2\beta x^\alpha \tan (\pi\alpha)\right];$
\item[{\rm(ii)}] $H_2(x)=\displaystyle M_2\frac{x^{1-\alpha}}{1+x^\alpha}\exp(\beta x^\alpha -bx^\gamma )\sin\left[b x^\gamma\tan(\pi \gamma)\right]$, where $b>0, \gamma\in (\alpha,1/2);$
\item[{\rm(iii)}]  $H_3(x)=\displaystyle M_3\frac{\sin[\beta x^\alpha\tan(\pi\alpha)-\pi \alpha]+x^\alpha\sin[\beta x^\alpha\tan(\pi\alpha)-2\pi \alpha ]}{1+x^\alpha},$
\end{enumerate} where $H_i(x)=0$ for $x<0$ and  constants $M_i$ are chosen in such a way that $\displaystyle\sup_{x\in \mathbb{R}}|H_i(x)|=1, \,i=1,2,3.$
\end{theorem}

\begin{proof} Since all functions $H_i$ satisfy $\sup_{x\in \mathbb{R}}|H_i(x)|=1$, what is left is to show that
 \begin{equation}\label{oh}\int_0^\infty x^k f(x)H_i(x)dx=0 \quad k\in \mathbb{N}_0,\;i=1,2,3.\end{equation}
 The expressions (i) - (iii)  are derived with the help of \cite[Example 3.2]{alea}. Here, we only have to check equalities \eqref{oh}. For this purpose, the identities below  (\cite[formulae 3.944, 9 and 10]{ryzhik}) will be used:

 \begin{equation}\label{ryzhik9}
\int_0^\infty x^{p-1}e^{-qx}\sin(qx\tan t)dx=\frac{\Gamma (p)}{q^p}\cos^pt\sin(pt), \;p,q>0, \,|t|<\pi/2
\end{equation}
 and
 \begin{equation}\label{ryzhik10}
\int_0^\infty x^{p-1}e^{-qx}\cos(qx\tan t)dx=\frac{\Gamma (p)}{q^p}\cos^pt\cos(pt), \;p,q>0, \,|t|<\pi/2.
\end{equation}
 Denote: $$J_i(k):=\frac{\beta +1}{\alpha \beta^2 M_i}\int_0^\infty x^kf(x)H_i(x)dx,\quad i=1,2,3.$$
 Then, the substitution $x\mapsto x^\alpha$ yields
$$J_1(k)=
 \frac{1}{\alpha} \int_0^\infty x^{(k+1)/\alpha-1}e^{-2\beta x}\sin (2\beta x\tan(\pi \alpha))dx.$$ Setting $p=(k+1)/\alpha, q=2\beta,$ and $t=\pi\alpha,$ one derives from \eqref{ryzhik9}
 $$J_1(k)=\frac{\Gamma (p)}{\alpha q^p}\cos^p(\pi\alpha)\sin((k+1)\pi)=0,\;\;k\in \mathbb{N}_0.$$
Observe that \eqref{ryzhik9} is applicable because $p,$ $q>0$ and $t=\pi\alpha\in (0,\pi/2)$ by the condition on $\alpha.$

 Likewise, to justify (ii), we write:
 $$J_2(k)=\frac{1}{\alpha}\int_0^\infty x^{(k+1)/\gamma-1}e^{-b x}\sin (b x\tan(\pi \alpha))dx.$$ This is an integral of the form \eqref{ryzhik9}, where  $p=(k+1)/\gamma, q=b,$ and $t=\pi\alpha$ and hence
 $J_2(k)=0$ as in the previous case.

 Finally, in the case (iii), integral $J_3(k)$ can be split as
 \begin{multline*}
\frac{1}{\alpha}\int_0^\infty x^{k+\alpha -1}e^{-\beta x^\alpha}\sin [ \beta x^\alpha\tan (\pi\alpha)-\pi\alpha ]dx + \\
 \frac{1}{\alpha}\int_0^\infty x^{k+2\alpha -1}e^{-\beta x^\alpha}\sin [\beta x^\alpha\tan (\pi\alpha)-2\pi\alpha]dx=:U(k)+V(k).
\end{multline*}
The same substitution $x\mapsto x^\alpha$ leads to:
 \begin{multline*}
U(k)=\cos(\pi \alpha)\int_0^\infty x^{k/\alpha}e^{-\beta x}\sin (\beta x\tan(\pi \alpha))dx \\
-\sin(\pi \alpha)\int_0^\infty x^{k/\alpha}e^{-\beta x}\cos (\beta x\tan(\pi \alpha))dx.
\end{multline*}
Applying formulae \eqref{ryzhik9} and \eqref{ryzhik10} with $p=k/\alpha +1, q=\beta,$ and $t=\pi\alpha,$ one derives that $U(k)=\frac{\Gamma (p)}{\alpha q^p}\cos^p(\pi\alpha)\sin(k\pi)=0.$ Similarly, with $p= p=k/\alpha +2, q=\beta,$ and $t=\pi\alpha,$ we obtain that
$V(k)=\frac{\Gamma (p)}{\alpha q^p}\cos^p(\pi\alpha)\sin(k\pi)=0.$

\end{proof}

\begin{corollary} Let $f$ be a PDF for $PL(\alpha,\beta)$ distribution with $\alpha<1/2.$ Then, the following sets are Stieltjes classes for $f$:
$$S_i=\{f_{\epsilon}(x):f_{\epsilon}(x)=f(x)\left[1+\epsilon H_i(x)\right],\ x\in\mathbb{R},\ \epsilon\in[-1,1]\},\quad i=1,2,3.$$
\end{corollary}

\section{Application to software metrics}

Software metrics are objective measurements of software products used to assess the quality of the products. These days, a variety of software metrics are being proposed related to different parameters such as the size (of software as a whole or size of its inherent classes and methods), complexity (of software system, classes, methods), internal and external quality  characteristics of a software system. 
See, for example, \cite{kitchen, deepti}. Correspondingly, ample amount of data on the values of software metrics were collected and, as a result, a statistical analysis of such data has become in demand within engineering studies. See, for example, \cite{ferreira, deepti} and \cite{mile} where one can find an extensive list of references. In some problems related to software metrics, such as creating catalogues for threshold values, it is important to find probability distributions which best fit the empirical data. In the literature, the two-parameter Weibull distribution has been indicated as a useful instrument for this purpose, while new distributions are being offered by statisticians aiming to provide better tools for specific practical problems.

In this section, we implement the power Lindley distribution to data arrays provided to the authors as a courtesy by M. Stojkovski \cite{mile}, who collected the data related to 17 unique categories and, in each category, calculated the values of the following 5 metrics:
\begin{itemize}
\item CBO (Coupling Between Objects)
\item DIT (Depth of Inheritance Tree)
\item NOC (Number Of Children)
\item NOM (Number Of Methods)
\item RFC (Response For Class)
\end{itemize}

In this article, the data related to DIT  and NOC  metrics are used. These metrics  were introduced  and investigated by Chidamber and Kemerer \cite{chidam1} in order to measure complexity and coupling. The other data sets available in \cite{mile} can be analyzed likewise.

In the next two examples, the MATHLAB software was used and the method of least squares was applied to fit the power Lindley density.

\begin{example}[DIT system metric] DIT represents the maximum length of the path, as a number of graph edges, from a node to the root of the inheritance tree. It is known that the greater DIT value is, the higher the complexity of a design becomes. 
The data collected in \cite{mile} can be summarized in Table \ref{table:dit}.

\begin{table}[!ht]
    \caption{DIT in system category}
		\label{table:dit}
    \centering
    \begin{tabular}{|c|c|} \hline
\text{Values} & \text{Frequencies}\\ \hline
0 & 35.45 \\ \hline
1 & 54.27 \\ \hline
2 & 7.94 \\ \hline
3 & 1.50 \\ \hline
4 & 0.77 \\ \hline
5 & 0.07 \\ \hline
\end{tabular}            
\end{table}
Using the method of least squares, these data were approximated by the power Lindley density with $\alpha=1.1913, \beta=1.6979.$ Also, for comparison, we used the fitted Weibull distribution found in \cite{mile} with the help of the EasyFit software. Also, the error of approximation in each case was obtained. Table \ref{table:ex1res} summarizes the results and Figure \ref{fig:ex1} shows the data along with the fitted curves.
\begin{table}[!ht]
    \caption{DIT in system category}
		\label{table:ex1res}
    \centering
    \begin{tabular}{|c|c|c|c|c|c|c|} \hline
    \multirow{2}{*}{Distribution} &
      \multicolumn{2}{c|}{Parameters} & 
			\multirow{2}{*}{Error} & \multirow{2}{*}{Mean} & \multirow{2}{*}{Median} & \multirow{2}{*}{$|\bar{X}-\text{Mean}|$} \\ \cline{2-3}
			& $\alpha$ & $\beta$ & & & & \\ \hline
$PL_{LS}$  & 1.1913 & 1.6979 & 0.0065 & 0.7923 & 0.6475 & 0.0150 \\ \hline
$W_{EF}$   & 1.3969 & 1.0044 & 0.0741 & 0.9158 & 0.7726 & 0.1385 \\ \hline
\end{tabular}            
\end{table}

\begin{figure}[!ht]
\centering
\includegraphics[width=100mm]{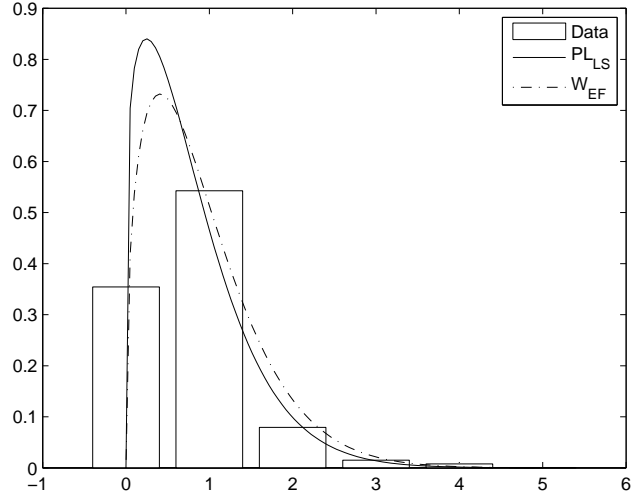}
\caption{Fitted distributions for DIT-system category}
\label{fig:ex1}
\end{figure}

\end{example}

\begin{example}[NOC system metric] 
NOC represents the number of immediate subclasses of a class in the hierarchy, measuring the number of subclasses inheriting the methods of the parent class. It is known that when NOC rises, so does re-use.
The highlights of the data collected in \cite{mile} appears in Table \ref{table:noc}.

\begin{table}[!ht]
    \caption{NOC in system category}
		\label{table:noc}
    \centering
    \begin{tabular}{|c|c|c|c|c|c|} \hline
\text{Value} & \text{Frequency} & \text{Value} & \text{Frequency} & \text{Value} & \text{Frequency} \\ \hline
0 & 92.21 &  7 & 0.09 & 14 & 0.04 \\ \hline
1 &  3.73 &  8 & 0.06 & 15 & 0.04 \\ \hline
2 &  1.99 &  9 & 0.11 & 17 & 0.02 \\ \hline
3 &  0.64 & 10 & 0.09 & 18 & 0.02 \\ \hline
4 &  0.32 & 11 & 0.02 & 19 & 0.04 \\ \hline
5 &  0.21 & 12 & 0.11 & 29 & 0.02 \\ \hline
6 &  0.19 & 13 & 0.04 &    &      \\ \hline
\end{tabular}            
\end{table}
It can be observed that the behavior of this data set is essentially different from that of DIT. The data set possesses strong right-skewed pattern, where the frequency of 0 dominates all of the other frequencies.

Like before, the method of least squares was applied and the outcomes along with the fitted Weibull distribution found in \cite{mile} by means of the EasyFit software are placed in Table \ref{table:ex2noc} and Figure \ref{fig:ex2fig1}.
\begin{table}[!ht]
    \caption{NOC in system category}
		\label{table:ex2noc}
    \centering
    \begin{tabular}{|c|c|c|c|c|c|c|} \hline
    \multirow{2}{*}{Distribution} &
      \multicolumn{2}{c|}{Parameters} & 
			\multirow{2}{*}{Error} & \multirow{2}{*}{Mean} & \multirow{2}{*}{Median} & \multirow{2}{*}{$|\bar{X}-\text{Mean}|$} \\ \cline{2-3}
			& $\alpha$ & $\beta$ & & & & \\ \hline
$PL_{LS}$  & 0.2750 & 3.6502 & 0.0001 & 0.2265 & 0.0053 & 0.1174 \\ \hline
$W_{EF}$   & 0.9499 & 1.0104 & 0.1136 & 1.0341 & 0.6869 & 0.9249 \\ \hline
\end{tabular}            
\end{table}

\begin{figure}
\centering
\includegraphics[width=120mm]{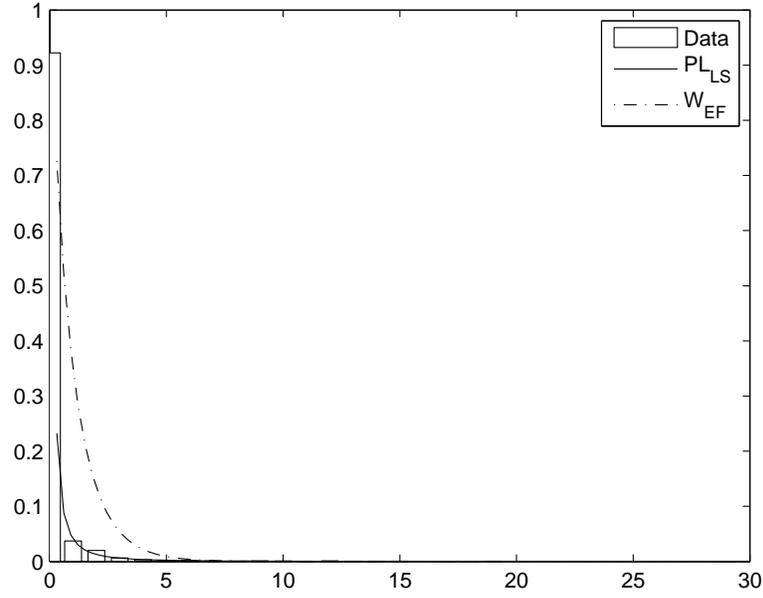}
\caption{Fitted distributions for NOC-system category}
\label{fig:ex2fig1}
\end{figure}

\begin{figure}[hbtp]
\begin{center}
\mbox{
\subfigure[Interval $0\leqslant x \leqslant 2$]{\includegraphics[width=60mm]{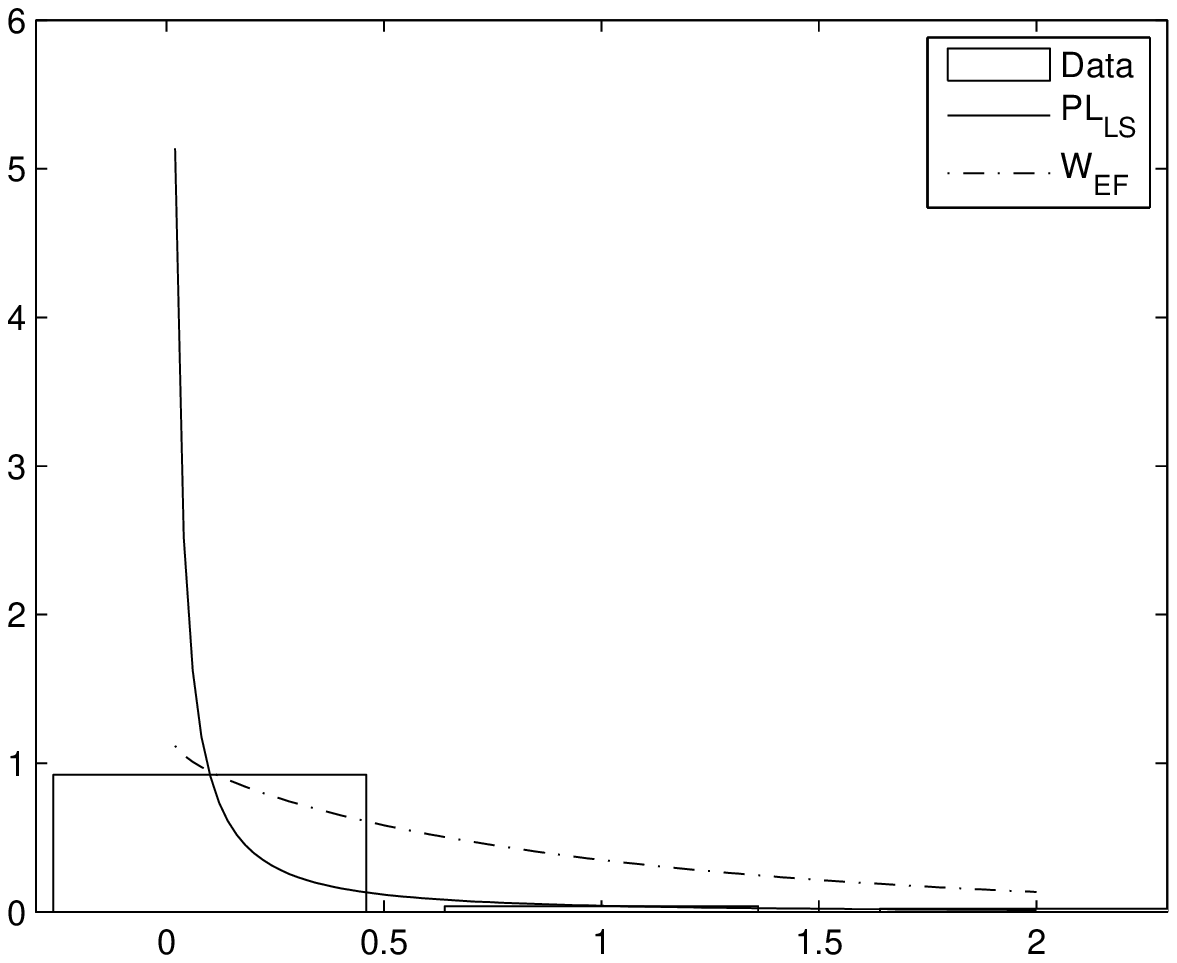}}\quad
\subfigure[Interval $2\leqslant x \leqslant10$]{\includegraphics[width=60mm]{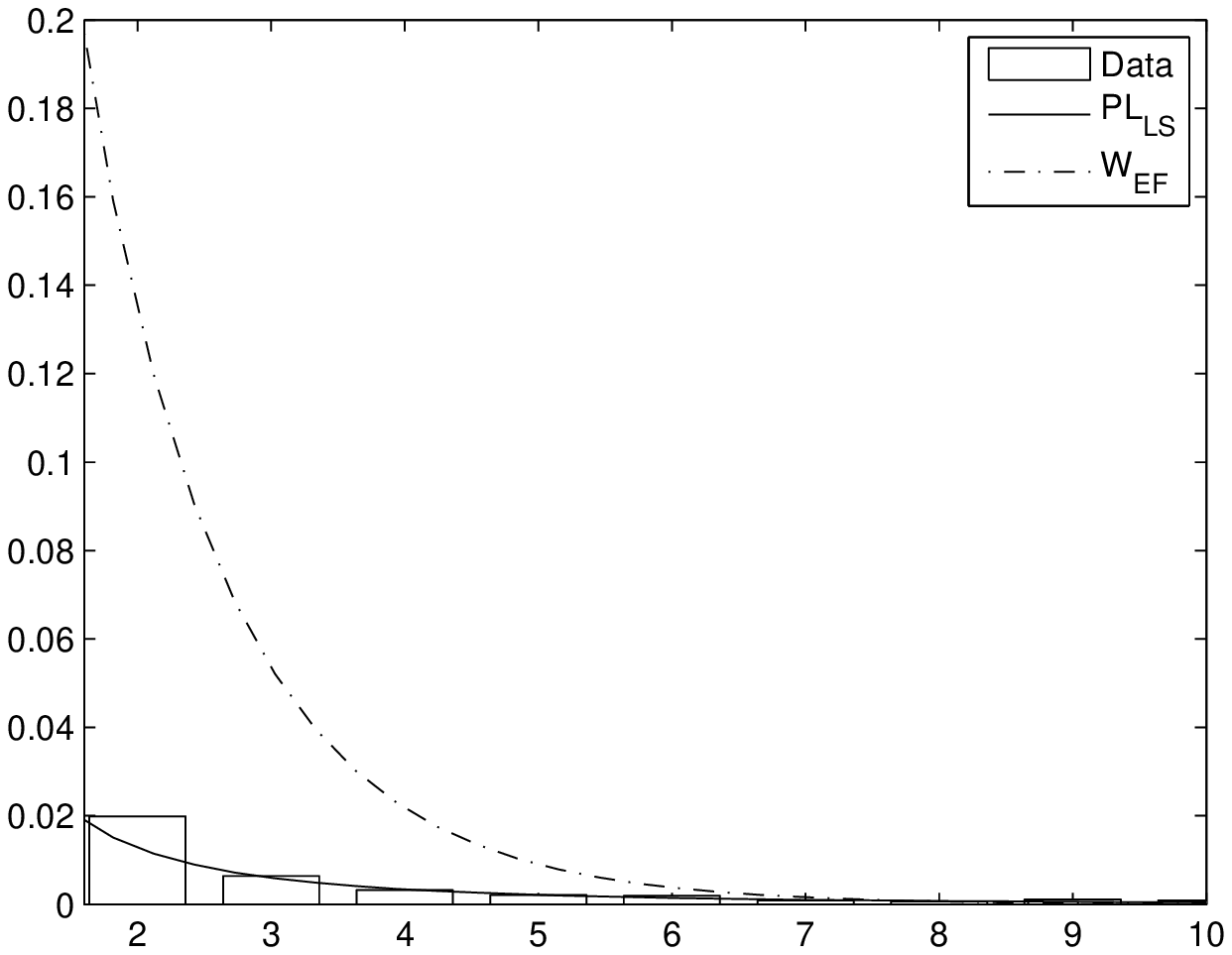}} 
}
\caption{Data and fitted densities on different intervals}
\end{center}
\end{figure}

\end{example}

\section{Conclusion}
This work is a continuation of the study on power Lindley distribution, 
initiated by M. E. Ghitany et al. in \cite{ghitanypower}. The goal of the current research is to obtain new results on the distribution and provide some novel applications. Since the power Lindley distribution becomes heavy-tailed when $\alpha <1$ - and, consequently, does not possess a moment-generating function - the examination  of its moment-(in)determinacy in this case has to be carried out. This is precisely the main result of this paper, stating that $PL(\alpha,\beta)$ distribution is moment-indeterminate if and only if $\alpha <1/2.$ Several Stieltjes classes have been constructed for this case.

Furthermore, this paper has discussed certain applications dealing with real data sets pertinent to the values of software metrics. Software metrics are currently a hot topic in the software engineering as they address quality standards followed by the software developers. The two-parameter Weibull distribution is commonly used to fit experimental data sets of software metrics. In this research, using the data collected in \cite{mile} for DIT and NOC metrics, it is shown that, for certain data sets, power Lindley distribution provides a better description of the data than Weibull distribution, not only for the light- but also for the heavy-tailed case. It has to be pointed out that both distributions are two-parameter, and therefore, similar in terms of complexity of the models. As for future work, it is planned to perform a similar data analysis for other software metrics and find new threshold values in collaboration with respective specialists.

\section*{Acknowledgements}
The authors express their sincere gratitude to Dr. Deepti Mishra (NTNU) for consulting them on the software metrics and to Mr. Mile Stojkovski for providing the collected data sets along with relevant references. Also, our thanks go to Mr. P. Danesh from the Atilim University Academic Writing and Advisory Center for his help in the presentation of the manuscript.

\end{document}